\documentclass[11pt]{article}
\usepackage{latexsym}

\ifx\pdftexversion\undefined
  \usepackage[dvips]{graphics}
\else
  \usepackage[pdftex]{graphics}
\fi
\font\german = eufm10 scaled\magstep1
\font\Cp = msbm10

\newcommand{\Rrr}{\hbox{\Cp R}}

\newcommand{\Ssss}{\hbox{\german S}}

\newcommand{\qed}{\mbox{$\Box$}\vspace{\baselineskip}}

\newenvironment{proof}{\noindent {\bf Proof:}}
                      {{\qed}}

\newenvironment{proof_special_}[1]{\noindent {\bf #1}}
                              {\vspace{-2mm}}

\newtheorem{theorem}{Theorem}[section]

\newtheorem{definition}[theorem]{Definition}
\newtheorem{corollary}[theorem]{Corollary}
\newtheorem{example}[theorem]{Example}

\newtheorem{continuation}[theorem]{Continuation of Example}

\newcommand{\hz}{\hat{0}}
\newcommand{\ho}{\hat{1}}

\newcommand{\um}{\underline{m}}

\font\Cp = msbm10

\begin{document}

\title{The M\"obius function of partitions with restricted block sizes}

\author{{\sc Richard EHRENBORG\thanks{Partially
supported by National Science
Foundation grant 0200624.}$\;$  
         and Margaret A.\ READDY}}

\date{
{\small
{\it  Department of Mathematics,
      University of Kentucky,
      Lexington, KY 40506 USA} \\[2 mm]
      Received 3 February 2006; accepted 30 August 2006
}
}

\maketitle

\begin{abstract}
The purpose of this paper is to compute the M\"obius function
of filters in the partition lattice formed
by restricting to partitions by type.
The M\"obius function is determined in terms of the
descent set statistics on permutations and the
M\"obius function of filters in
the lattice of integer compositions.
When the underlying integer partition is a knapsack
partition,
the M\"obius function on integer compositions
is determined by a topological argument.
In this proof the permutahedron
makes a cameo appearance.

\vspace*{2 mm}


\vspace*{2 mm}

\end{abstract}

\vspace*{-2mm}

\section{Introduction}
\setcounter{equation}{0}

\vspace*{-1mm}

Sylvester~\cite{Sylvester} initiated
the study of subposets of the
partition lattice.
He proved that the M\"obius function of 
the poset of set partitions
where each block has even cardinality is given
by every other  Euler number, also known as the tangent numbers.
Recall the Euler numbers enumerate  alternating permutations.
Stanley~\cite{Stanley_e_s} extended this result
by considering set partitions where each block has cardinality divisible
by $r$. 
In this case the M\"obius function is given
by the number of permutations with descent set
$\{r, 2r, 3r, \ldots\}$.
In this paper we continue to 
explore the 
connection between 
partition lattices and permutation statistics.

To each set partition we can assign a type, which is
the integer partition consisting of the
multiset of cardinalities of the blocks.
Given a set $F$ of integer partitions it is then natural
to ask for the M\"obius function of the associated poset of set partitions
whose types belong to $F$.
We slightly modify this by instead working
with 
pointed set partitions
and pointed integer partitions
and letting $F$ be a filter 
in the poset of
pointed integer partitions.

In Theorem~\ref{theorem_main} we consider
pointed set partitions whose types
belong to given filter $F$ of pointed integer partitions.
The question of computing the M\"obius function
is reduced to the descent set statistics
and M\"obius functions in the smaller and more tractable 
lattice of integer compositions.
In Section~\ref{section_knapsack}
we consider knapsack partitions,
a notion motivated by a well-known cryptosystem.
Theorem~\ref{theorem_M}
allows us to determine the M\"obius function
of the integer composition lattice.
We  obtain explicit expressions for the sought-after M\"obius function
in Theorem~\ref{theorem_three}.

\vspace*{-2mm}

\section{Pointed integer partitions, set partitions
         and compositions}
\setcounter{equation}{0}

\vspace*{-1mm}

Recall that an integer partition
$\lambda =  \{\lambda_{1}, \ldots, \lambda_{k}\}$
of a non-negative integer $n$ is a multiset of positive
integers having sum $n$,
that is,
$n = \lambda_{1} + \cdots + \lambda_{k}$.

\begin{definition}
Let $n$ be a non-negative integer.
A pointed integer partition
of $n$
is a pair
$\{\lambda, \um\} =
    \{\lambda_{1}, \ldots, \lambda_{k}, \um\}$
where
$m$ is a non-negative integer
and 
$\lambda =  \{\lambda_{1}, \ldots, \lambda_{k}\}$ 
is an integer partition of $n-m$.
The integer $m$ is called the pointed part.
It  is underlined to distinguish it from the
other parts of the partition.
\end{definition}
Denote by  $I_{n}^{\bullet}$ 
the collection of all pointed integer partitions
of the non-negative integer $n$.
Partially order the set $I_{n}^{\bullet}$
by the two cover relations
$$
    \{\lambda_{1}, \ldots, \lambda_{k-2},
         \lambda_{k-1}, \lambda_{k}, \um\}
       \prec 
    \{\lambda_{1}, \ldots, \lambda_{k-2},
         \lambda_{k-1} + \lambda_{k}, \um\}
$$
and
$$
    \{\lambda_{1},
         \ldots, \lambda_{k-1}, \lambda_{k}, \um\}
       \prec 
    \{\lambda_{1},
         \ldots, \lambda_{k-1}, \underline{\lambda_{k} + m}\} .
$$
In words, by adding two parts together we go up in the order.
If one of the parts is the pointed part then the sum becomes the
pointed part in the new pointed partition.
Observe that the poset $I_{n}^{\bullet}$ is not a lattice for $n \geq 3$.

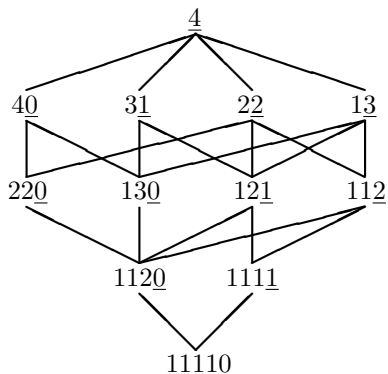
\begin{figure}
\setlength{\unitlength}{0.5mm}
\begin{center}
\begin{picture}(90,90)(0,0)

\thicklines

\put(37,0){\small $1 1 1 1 \underline{0}$}

\put(45,6){\line(-1,1){15}}
\put(45,6){\line(1,1){15}}

\put(23,23){{\small $1 1 2 \underline{0}$}}
\put(53,23){{\small $1 1 1 \underline{1}$}}

\put(30,29){\line(-2,1){30}}
\put(30,29){\line(0,1){15}}
\put(30,29){\line(2,1){30}}
\put(30,29){\line(4,1){60}}
\put(60,29){\line(0,1){15}}
\put(60,29){\line(2,1){30}}

\put(-5,46){\small $2 2 \underline{0}$}
\put(25,46){\small $1 3 \underline{0}$}
\put(55,46){\small $1 2 \underline{1}$}
\put(85,46){\small $1 1 \underline{2}$}

\put(0,52){\line(0,1){15}}
\put(0,52){\line(4,1){60}}
\put(30,52){\line(-2,1){30}}
\put(30,52){\line(0,1){15}}
\put(30,52){\line(4,1){60}}
\put(60,52){\line(-2,1){30}}
\put(60,52){\line(0,1){15}}
\put(60,52){\line(2,1){30}}
\put(90,52){\line(-2,1){30}}
\put(90,52){\line(0,1){15}}

\put(-4,69){\small $4 \underline{0}$}
\put(26,69){\small $3 \underline{1}$}
\put(56,69){\small $2 \underline{2}$}
\put(86,69){\small $1 \underline{3}$}

\put(0,75){\line(3,1){45}}
\put(30,75){\line(1,1){15}}
\put(60,75){\line(-1,1){15}}
\put(90,75){\line(-3,1){45}}

\put(43,92){\small $\underline{4}$}

\end{picture}
\end{center}
\caption{The poset $I_{4}$ of pointed partitions of the integer $4$.}
\label{figure_pointed_partitions_4}
\end{figure}

\begin{definition}
A pointed set partition $\pi = (\sigma,Z)$
of a finite set $S$ 
consists of a subset $Z$ of $S$ and a
partition $\sigma$ of the set difference $S - Z$.
\end{definition}
We call the subset $Z$ the {\em zero block} of
the pointed set partition $\pi$.
Moreover, we denote the number of blocks
(including the zero block) of $\pi$ by $|\pi|$.
Let $\Pi_{n}^{\bullet}$ be the poset of all
pointed set partitions on the set $\{1, \ldots, n\}$,
where the partial order is given by
refinement. That is,
for two pointed set partitions $\pi$ and $\pi^{\prime}$,
we have that $\pi \leq \pi^{\prime}$ if every block
of $\pi$ is contained in some block
(possibly the zero block) of $\pi^{\prime}$
and
the zero block of $\pi$ is contained in 
the zero block of $\pi^{\prime}$.
Observe that $\Pi_{n}^{\bullet}$ is isomorphic to
the partition lattice $\Pi_{n+1}$
by inserting the element $n+1$ into the zero block
and letting the zero block be an ordinary block of the partition.

We define the {\em type} of a pointed partition $(\sigma,Z)$
to be the pointed integer partition $\{\lambda,\um\}$
where~$m$ is the cardinality of the zero block $Z$
and $\lambda$ is the multiset of the sizes of the blocks of $\sigma$,
that is,
$\lambda = \{|B| \: : \: B \in \sigma\}$.

\begin{definition}
Let $n$ be a non-negative integer.
A pointed integer composition of $n$
is a list $\vec{c} = (c_{1}, \ldots, c_{k-1}, \underline{c_{k}})$
of non-negative integers with sum $n$
where $c_{1}$ through $c_{k-1}$ are required to be positive.
\end{definition}
Note that the last entry $c_{k}$ is allowed to be $0$,
so we underline it to distinguish it from the other entries.
Let $C_{n}^{\bullet}$ be the collection of all
pointed compositions of $n$. 
Partially order
the elements of $C_{n}^{\bullet}$ 
by the cover relations
$$
   (c_{1}, \ldots, c_{j-1}, c_{j}, c_{j+1}, c_{j+2}, \ldots, c_{k-1},
                                                  \underline{c_{k}})
       \prec 
   (c_{1}, \ldots, c_{j-1}, c_{j} + c_{j+1}, c_{j+2}, \ldots, c_{k-1},
                                                  \underline{c_{k}})
,  $$
$$
   (c_{1}, \ldots, c_{k-2}, c_{k-1}, \underline{c_{k}})
       \prec 
   (c_{1}, \ldots, c_{k-2}, \underline{c_{k-1} + c_{k}})
.  $$
That is, the cover relation occurs by adding two adjacent entries
of the composition.
Observe that the poset~$C_{n}^{\bullet}$ is isomorphic to the
Boolean algebra on $n$ elements.

For a permutation $\tau$ in
the symmetric group on $n$ elements $\Ssss_{n}$,
the {\em descent set} is a subset of $\{1,\ldots,n-1\}$ defined as
$\{i \: : \: \tau(i) > \tau(i+1)\}$.
Given a permutation $\tau$ in $\Ssss_{n}$
with descent set
$\{s_{1}, s_{2}, \ldots, s_{k-1}\}$
where $s_{1} < s_{2} < \cdots < s_{k-1}$,
define the descent composition of $\tau$
to be
$(s_{1}, s_{2}-s_{1}, \ldots, s_{k-1}-s_{k-2}, n-s_{k-1})$.
For a composition $\vec{c} = (c_{1}, \ldots, c_{k})$ of $n$
let $\beta(\vec{c})$ denote the number of permutations in
the symmetric group $\Ssss_{n}$ with descent composition $\vec{c}$,
that is, with the descent set
$\{c_{1}, c_{1}+c_{2}, \ldots, c_{1} + \cdots + c_{k-1}\}$.
For a pointed composition
$\vec{c} = (c_{1}, \ldots, c_{k-1}, \underline{c_{k}})$
with $c_{k}$ strictly greater than $0$,
let $\beta(\vec{c})$ be as in the non-pointed composition case.
If the pointed composition 
$\vec{c} = (c_{1}, \ldots, c_{k-1}, \underline{c_{k}})$
satisfies $k \geq 2$ and $c_{k} = 0$,
let $\beta(\vec{c}) = 0$.
Lastly, for $\vec{c} = (\underline{0})$ let $\beta(\vec{c}) = 1$.

A uniform way to view the descent composition
statistic independently of the last part
is the following.
For a pointed composition 
$\vec{c} = (c_{1}, \ldots, c_{k-1}, \underline{c_{k}})$
let 
$\vec{d}$ be the composition
where the last part is incremented by one, that is,
$\vec{d} = (c_{1}, \ldots, c_{k-1}, c_{k}+1)$.
Then $\beta(\vec{c})$
is the number of permutations~$\tau$ in 
$\Ssss_{n+1}$ having
descent composition $\vec{d}$ and
satisfying $\tau(n+1) = n+1$.

The type of an integer composition
$(c_{1}, \ldots, c_{k-1}, \underline{c_{k}})$
is the pointed integer partition
$\{c_{1}, \ldots, c_{k-1}, \underline{c_{k}}\}$.
Observe that the last part of the composition
becomes the pointed part.

We remark that the three posets 
$I_{n}^{\bullet}$,
$\Pi_{n}^{\bullet}$,
and
$C_{n}^{\bullet}$
are graded, and hence ranked.
Throughout we will denote the rank function of
a graded poset by $\rho$
and use
$\rho(x,y)$ to denote the
rank difference
$\rho(x,y) = \rho(y) - \rho(x)$.

\vspace*{-2mm}

\section{The M\"obius function of restricted partitions}
\setcounter{equation}{0}

\vspace*{-1mm}

Recall that a {\em filter} $F$ (also known as an upper order ideal)
in a poset $Q$ is a subset of $Q$ such that if
$x \leq y$ and $x$ belongs to $F$
then $y$ belongs to $F$. 
For $S$ a subset of the poset $Q$,
the filter generated by~$S$ is given by
$\{y \in Q \: : \: \exists \, x \in S \mbox{ such that }
                        x \leq y\}$.
Note that
if $f$ is an order preserving map from a poset $P$ to a poset $Q$
and $F$ is a filter of $Q$ then the inverse
image $f^{-1}(F)$ is a filter of $P$.
For further information about posets,
we refer the reader to Stanley's treatise~\cite{Stanley_EC_I}.

Let $F$ be a filter in the pointed integer partition poset $I_{n}^{\bullet}$.
Let $\Pi_{n}^{\bullet}(F)$ be the filter of the pointed set partition lattice
$\Pi_{n}^{\bullet}$
consisting of all set partitions having their types belonging to $F$,
that is, 
$$   \Pi_{n}^{\bullet}(F)
   =
       \{ \pi \in \Pi_{n}^{\bullet} \:\: : \:\: {\rm type}(\pi) \in F \}
       . $$
Similarly, define 
$C_{n}^{\bullet}(F)$ to be the filter of pointed compositions
having types belonging to $F$,
that is, 
$$   C_{n}^{\bullet}(F)
   =
       \{ \vec{c} \in C_{n}^{\bullet} \:\: : \:\: {\rm type}(\vec{c}) \in F \}   .
$$
Observe that both
$\Pi_{n}^{\bullet}(F)$ and 
$C_{n}^{\bullet}(F)$ are join semi-lattices.
Hence after adjoining a minimal element~$\hz$
we have that both
$\Pi_{n}^{\bullet}(F) \cup \{\hz\}$
and 
$C_{n}^{\bullet}(F) \cup \{\hz\}$
are lattices.

\begin{theorem}
Let $F$ be a filter of the pointed integer partition poset
$I_{n}^{\bullet}$.
Then the M\"obius function of the
filter~$\Pi_{n}^{\bullet}(F)$ with a minimal element $\hz$
adjoined is given by
\begin{equation}
      \mu\left( \Pi_{n}^{\bullet}(F) \cup \{\hz\} \right)
   =
      \sum_{\vec{c} \in C_{n}^{\bullet}(F)}
             (-1)^{\rho(\vec{c},\ho)}
           \cdot
             \mu_{C_{n}^{\bullet}(F) \cup \{\hz\}}(\hz,\vec{c})
           \cdot
             \beta(\vec{c})              .  
\label{equation_main}
\end{equation}
\label{theorem_main}
\end{theorem}
Although $C_{n}^{\bullet}(F) \cup \{\hz\}$ is not necessarily
graded, the expression 
$\rho(\vec{c},\ho)$ appearing in the statement of
Theorem~\ref{theorem_main} is well-defined since the
interval 
$[\vec{c},\ho]$ is itself graded.

\begin{proof_special_}{{Proof of Theorem~\ref{theorem_main}:}}
For a pointed composition
$\vec{c} = (c_{1}, \ldots, c_{k-1}, \underline{c_{k}})$
of a non-negative integer~$n$
recall that the multinomial coefficient ${n \choose \vec{c}}$
is defined as $n!/(c_{1}! \cdots c_{k}!)$. Observe that
${n \choose \vec{c}}$ counts the number of permutations
in $\Ssss_{n}$ having descent set contained in
the set $\{c_{1}, c_{1}+c_{2}, \ldots, c_{1} + \cdots + c_{k-1}\}$.
By the principle of inclusion and exclusion, we have that
$\beta(\vec{c}) = \sum_{\vec{c} \leq \vec{d}} \;
                     (-1)^{\rho(\vec{c},\vec{d})} 
                        \cdot {n \choose \vec{d}}$.
Thus the right-hand side 
of~(\ref{equation_main}) is given by
$$
      \sum_{\vec{c} \in C_{n}^{\bullet}(F)}
             (-1)^{\rho(\vec{c},\ho)}
           \cdot
             \mu_{C_{n}^{\bullet}(F) \cup \{\hz\}}(\hz,\vec{c})
           \cdot
             \beta(\vec{c})
\hspace*{70 mm}
$$
\begin{eqnarray*}
  & = &
      \sum_{\vec{c} \in C_{n}^{\bullet}(F)}
      \sum_{\vec{c} \leq \vec{d}} 
             (-1)^{\rho(\vec{d},\ho)}
           \cdot
             \mu_{C_{n}^{\bullet}(F) \cup \{\hz\}}(\hz,\vec{c})
           \cdot
             {n \choose \vec{d}}    \\
  & = &
      \sum_{\vec{d} \in C_{n}^{\bullet}(F)}
             (-1)^{\rho(\vec{d},\ho)}
           \cdot
             {n \choose \vec{d}}
           \cdot
      \sum_{\hz < \vec{c} \leq \vec{d}}
             \mu_{C_{n}^{\bullet}(F) \cup \{\hz\}}(\hz,\vec{c})  \\
  & = &
   -
      \sum_{\vec{d} \in C_{n}^{\bullet}(F)}
             (-1)^{\rho(\vec{d},\ho)}
           \cdot
             {n \choose \vec{d}}    .
\end{eqnarray*}
The last sum can be viewed as follows.
An {\em ordered set partition} is a partition
in $\Pi_{n}^{\bullet}$
with an ordering of the blocks
such that the last block is the zero block.
The type of an ordered
set partition is the pointed composition
where one lists the size of each block
in their given order.
Note that given a composition~$\vec{d}$ of $n$,
there are ${n \choose \vec{d}}$
ordered set partitions with $\vec{d}$ as its type. Hence we have that
\begin{eqnarray*}
\hspace*{20 mm}
   -
      \sum_{\vec{d} \in C_{n}^{\bullet}(F)}
             (-1)^{\rho(\vec{d},\ho)}
           \cdot
             {n \choose \vec{d}}    
  & = &
      \sum_{{\pi \mbox{ \tiny ordered set partition}}
             \atop
            {\mbox{\tiny type}(\pi) \in C_{n}^{\bullet}(F)}}
              (-1)^{|\pi|}    \\
  & = &
      \sum_{\pi \in \Pi_{n}^{\bullet}(F)}
              (-1)^{|\pi|}
           \cdot
              (|\pi|-1)!      \\
  & = &
    -
      \sum_{\pi \in \Pi_{n}^{\bullet}(F)}
              \mu_{\Pi_{n}^{\bullet}(F) \cup \{\hz\}}(\pi,\ho) \\
  & = &
              \mu(\Pi_{n}^{\bullet}(F) \cup \{\hz\})   .
\hspace*{20 mm}
\hspace*{20 mm} \qed
\end{eqnarray*}
\end{proof_special_}

\begin{example}
{\rm
We have the following identity
connecting the Stirling numbers of the second kind
with the Eulerian numbers:
\begin{equation}
    - \sum_{j=1}^{k} 
          (-1)^{j-1} \cdot (j-1)!
        \cdot
          S(n+1,j)
  =
      (-1)^{k}
    \cdot
      \sum_{j=1}^{k} 
          {{n-j} \choose {n-k}}
        \cdot
          A(n,j)     .
\label{equation_Eulerian_Stirling}
\end{equation}
Recall that the Stirling number of the second
kind $S(n,j)$ counts the number of set partitions
of an $n$-element set into $j$ parts,
whereas the Eulerian number $A(n,j)$
counts the number of permutations in~$\Ssss_{n}$
with $j-1$ descents. 
To prove
equation~(\ref{equation_Eulerian_Stirling})
let $F$ be the filter of $I_{n}^{\bullet}$ consisting
of all pointed integer partitions with at most $k$ parts.
Then the filter $\Pi_{n}^{\bullet}(F)$ consists of
all pointed set partitions with at most $k$ parts.
Using the Stirling numbers of the second kind,
we can write the M\"obius function of
$\Pi_{n}^{\bullet}(F) \cup \{\hz\}$ to be the left-hand side of
equation~(\ref{equation_Eulerian_Stirling}).
Similarly, the filter $C_{n}^{\bullet}(F)$ consists of
all pointed compositions
of $n$ with at most $k$ parts.
Let $\vec{c}$ be a composition with $j$ parts.
Then the interval
$[\hz,\vec{c}]$ in the poset
$C_{n}^{\bullet}(F) \cup \{\hz\}$
is a rank-selected Boolean algebra,
more specifically,
the Boolean algebra~$B_{n-j+1}$
with ranks $1$ through $n-k$ removed.
Hence the M\"obius function
$\mu_{C_{n}^{\bullet}(F) \cup \{\hz\}}(\hz,\vec{c})$
is given by
$(-1)^{k-j+1} \cdot {{n-j} \choose {n-k}}$.
Summing over all pointed compositions $\vec{c}$
consisting of $j$ parts in the right-hand side
in Theorem~\ref{theorem_main},
we obtain
$(-1)^{k} \cdot {{n-j} \choose {n-k}}$
times the number of permutations in $\Ssss_{n}$
with $j-1$ descents. 
Hence the identity follows.
When $k=n$ this identity states that
$\mu(\Pi_{n+1}) = (-1)^{n} \cdot n!$
since
$\Pi_{n}^{\bullet}(F) \cup \{\hz\} \cong \Pi_{n}^{\bullet} \cong \Pi_{n+1}$.
}
\end{example}

We have the following corollary.
This result was proved with different techniques in~\cite{Ehrenborg_Readdy}.

\begin{corollary}
Let $n = r \cdot p + m$.
Let $\Pi_{n}^{\bullet,r,m}$ be all the partitions in 
$\Pi_{n}^{\bullet}$ where the zero block has cardinality
at least $m$ and the remaining blocks have
cardinality divisible by $r$.
Then
the M\"obius function of
$\Pi_{n}^{\bullet,r,m} \cup \{\hz\}$ is given by
$$   
      \mu\left(\Pi_{n}^{\bullet,r,m} \cup \{\hz\}\right)
   =
      (-1)^{p + 1} 
    \cdot 
      \beta(\underbrace{r,r, \ldots, r}_{p}, \um)
  . $$
\label{corollary_Dowling}
\end{corollary}
\begin{proof}
Let $F$ be the filter of $I_{n}^{\bullet}$ generated by
the pointed partition
$\{r,r, \ldots, r, \um\}$.
The poset $\Pi_{n}^{\bullet,r,m}$
is exactly the filter $\Pi_{n}^{\bullet}(F)$.
Observe that $C_{n}^{\bullet}(F) \cup \{\hz\}$ has a unique atom,
namely
the pointed composition $(r,r, \ldots, r, \um)$.
Thus the M\"obius function
$\mu_{C_{n}^{\bullet}(F) \cup \{\hz\}}(\hz,\vec{c})$ is
non-zero only for this composition.
Hence the summation in the right-hand side
of~(\ref{equation_main})
has only one term, namely
$(-1)^{p + 1} \cdot \beta(r,r, \ldots, r, \um)$.
\end{proof}

By setting $m=r-1$ in the previous corollary
and using the bijection between
$\Pi^{\bullet}_{n}$ and $\Pi_{n+1}$,
we obtain the following corollary due to
Stanley~\cite{Stanley_e_s,Stanley_EC_II}.
\begin{corollary}[Stanley]
Let $n = r \cdot p$ and let $\Pi_{n}^{r}$ denoted
the $r$-divisible lattice, that is, all partitions
on $n$ elements where the cardinality of each block is
divisible by $r$ and a minimal element
$\hz$ is adjoined.
Then the M\"obius function of $\Pi_{n}^{r}$,
$\mu(\Pi_{n}^{r})$,
is given by the sign $(-1)^{p}$
times the
number of permutations~$\tau$ in $\Ssss_{n}$
with descent set
$\{r, 2r, \ldots, n-r\}$
and $\tau(n) = n$.
\label{corollary_r_divisible_lattice}
\end{corollary}

\vspace*{-6mm}

\section{Knapsack partitions}
\label{section_knapsack}
\setcounter{equation}{0}

\vspace*{-1mm}

Let $\lambda = \{e_{1}^{m_{1}}, e_{2}^{m_{2}}, \ldots, e_{q}^{m_{q}}\}$
be a partition, that is, a multiset of positive integers, 
where $m_{i}$ denotes the multiplicity of the element $e_{i}$.
We tacitly assume that all the $e_{i}$'s are distinct,
that is, $e_{i} \neq e_{j}$ for $i \neq j$.
Since there are $\prod_{i=1}^{q} (m_{i} + 1)$ multi-subsets $\mu$
of $\lambda$, the following inequality holds:
\begin{equation}
       \left| \left\{ \sum_{e \in \mu} e \:\: : \:\:
                \mu \subseteq \lambda \right\} \right|
   \leq
      \prod_{i=1}^{q} (m_{i} + 1)  .  
\label{equation_definition_knapsack}
\end{equation}
When equality holds in~(\ref{equation_definition_knapsack}),
we call
$\lambda$ {\em a knapsack partition.}
Observe this is equivalent to 
that each integer in the set on the left-hand side
of~(\ref{equation_definition_knapsack})
has a unique representation as a sum of elements from the
multiset $\lambda$.
When all the entries in $\lambda$ are distinct,
that is, $\lambda$ is a set, this definition
reduces to the usual notion
of a knapsack system appearing in cryptography.

\begin{example}
(a) If
$e_{1}, \ldots, e_{q}, m_{1}, \ldots, m_{q}$
are positive integers 
satisfying the inequality
$\sum_{i=1}^{j-1} m_{i} \cdot e_{i} \leq e_{j}$
for all $j = 2, \ldots, q$
then
$\{e_{1}^{m_{1}}, \ldots, e_{q}^{m_{q}}\}$
is a knapsack partition.

\noindent
(b)
If $\{\lambda_{1}, \ldots, \lambda_{p}\}$
is a knapsack partition, $q$ a prime greater than
the sum $\lambda_{1} + \cdots + \lambda_{p}$
and $j$ a positive integer less than $q$, then
$\{j \cdot \lambda_{1} \bmod q, \ldots, j \cdot \lambda_{p} \bmod q\}$
is a knapsack partition.
\end{example}
A pointed integer partition
$\{\lambda,\um\}$ is called a 
{\em pointed knapsack partition} if
$\lambda$ is a knapsack partition.

In this section we consider the filter $F$
generated by a pointed knapsack partition $\{\lambda,\um\}$.
We determine the M\"obius function
of the poset $C_{n}^{\bullet}(F) \cup \{\hz\}$
and thus obtain an explicit formula 
for
$\mu(\Pi_{n}^{\bullet}(F) \cup \{\hz\})$.
% In this section we determine the M\"obius function
% of the filter $C_{n}^{\bullet}(F)$ in the poset
% of pointed compositions of $n$
% when the filter $F$ in the pointed integer partition
% poset $I_{n}^{\bullet}$ is generated
% by a pointed knapsack partition $\{\lambda,\um\}$.
Before proceeding
one more definition is needed.
Let $V(\lambda,\um) = V$ be the collection of all
pointed compositions
$\vec{c} = (c_{1}, \ldots, c_{k-1}, \um)$ in
the filter $C_{n}^{\bullet}(F)$
such that when each $c_{i}$, $1 \leq i \leq k-1$, is expressed
as a sum of parts of $\lambda$, the summands for each $c_{i}$
are distinct.

\begin{example}
{\rm
For the pointed knapsack partition 
$\{1,1,1,4,\um\}$
the set $V$ of pointed compositions is
\begin{eqnarray*}
    V
      & = &
    \{(1,1,1,4,\um), (1,1,5,\um), (1,1,4,1,\um), (1,5,1,\um),  \\
      &   &
     \:\:\:\:
          (1,4,1,1,\um), (5,1,1,\um), (4,1,1,1,\um) \}  .  
\end{eqnarray*}
Observe the composition $(2,1,4,\um)$ does not belong
to $V$ since $2$ is the sum of two equal parts.
}
\label{example_knapsack}
\end{example}

\begin{example}
{\rm
For the pointed knapsack partition 
$\{r, r, \ldots, r, \um\}$
the set $V$ only consists
of the pointed composition
$(r, r, \ldots, r, \um)$.
}
\label{example_knapsack_r_r}
\end{example}

The ordered partition lattice $Q_{p}$ consists of all 
ordered partitions of the set $\{1, \ldots, p\}$
together with a minimal element $\hz$ adjoined.
The cover relation in $Q_{p}$ is to merge two adjacent blocks.
The ordered partition lattice is
isomorphic to
the face lattice of the
$(p-1)$-dimensional permutahedron;
see for instance~\cite{Billera_Sarangarajan}
or Exercise~2.9 in~\cite{Oriented_Matroids}.
Hence $Q_{p}$ is Eulerian and has
M\"obius function given by
$\mu_{Q_{p}}(x,y) = (-1)^{\rho(x,y)}$.

\begin{theorem}
Let $F$ be the filter in the pointed integer
partition poset $I_{n}^{\bullet}$ generated by 
the pointed knapsack partition
$\{\lambda,\um\}
   = \{\lambda_{1}, \ldots, \lambda_{p}, \um\}$
of the integer $n$.
Let
$\vec{c} = (c_{1}, \ldots, c_{k-1}, \underline{c_{k}})$
be a pointed composition in 
the filter $C_{n}^{\bullet}(F)$.
Then the M\"obius function
$\mu_{C_{n}^{\bullet}(F) \cup \{\hz\}}(\hz,\vec{c})$ is given
by
$$
 \mu(\hz,\vec{c})
      =
 \left\{ \begin{array}{c l}
    (-1)^{p-k}   & \mbox{ if } \vec{c} \in V, \\
    0            & \mbox{ otherwise.}
         \end{array} \right.  $$
\label{theorem_M}
\end{theorem}
\begin{proof}
Consider the lattice
$C_{n}^{\bullet}(F) \cup \{\hz\}$ and
let $\vec{c}$ be a pointed composition in $C_{n}^{\bullet}(F)$.
Assume that the pointed part of
the composition $\vec{c}$ is greater than $m$.
The interval $[\hz,\vec{c}]$ in
$C_{n}^{\bullet}(F) \cup \{\hz\}$ is itself a lattice
and each atom in this interval has pointed part equal to $m$.
The join of these atoms also has pointed part equal to $m$,
so the join cannot not equal the composition $\vec{c}$.
By Corollary~3.9.5 in~\cite{Stanley_EC_I},
the M\"obius function vanishes, that is, $\mu(\hz,\vec{c}) = 0$.

Let $P$ be the ideal of $C_{n}^{\bullet}(F)$ 
consisting of all pointed compositions with pointed part $m$.
Observe that $P$ has the maximal element
$(\lambda_{1} + \cdots + \lambda_{p}, \um)$.
Hence $P$ is a finite join semi-lattice
and thus $P \cup \{\hz\}$ is lattice.

We first consider the case when all the parts
of $\lambda$ are distinct.
In this case $P = V$.
There is an isomorphism between
the lattice $P \cup \{\hz\}$ and the ordered partition lattice
$Q_{p}$.
The isomorphism $f$ sends the ordered partition
$(B_{1}, B_{2}, \ldots, B_{k})$ to the
pointed composition
$$    f((B_{1}, B_{2}, \ldots, B_{k}))
    =
      \left( \sum_{i \in B_{1}} \lambda_{i},
             \sum_{i \in B_{2}} \lambda_{i},
                  \ldots,
             \sum_{i \in B_{k}} \lambda_{i} , \um \right)  , $$
and $f(\hz) = \hz$.
Hence we conclude that the
M\"obius function is given by
$\mu(\hz,\vec{c}) = (-1)^{p-k}$.

For the general case we allow the partition 
to contain multiple entries. The ideal $P$ is then
isomorphic to a filter of the ordered partition lattice $Q_{p}$.
Namely, let $R$ be the collection of all ordered
partitions 
$(B_{1}, B_{2}, \ldots, B_{k})$ such that
if $\lambda_{i} = \lambda_{j}$ for $1 \leq i < j \leq p$
then the elements $i$ and $j$ either appear in the same block
or the element $i$ appears in a block before the block containing
the element~$j$. It is straightforward to verify
that $R$ is a filter of $Q_{p}$
and that the map $f$ is again an isomorphism,
this time from $R$ to $P$.

Recall that the boundary of the dual of
the permutahedron is a simplicial
complex $\Delta_{p}$. We view this simplicial complex as
the $(p-1)$-dimensional sphere in $\Rrr^{p}$
cut by the ${p \choose 2}$  hyperplanes
$x_{i} = x_{j}$ for $1 \leq i < j \leq p$.
The geometric picture is the sphere $S^{p-1}$ with
${p \choose 2}$ great spheres on it.
Let ${\cal H}$ be the hyperplane arrangement
in $\Rrr^{p}$ given by
the collection of
the hyperplanes
$x_{i} = x_{j}$
when
$\lambda_{i} = \lambda_{j}$.
Let~${\cal C}$ be the chamber
${\cal C}
  =
    \{(x_{1}, \ldots, x_{p})
           \:\: : \:\:
       x_{i} \leq x_{j} \mbox{ if } i < j \mbox{ and }
                     \lambda_{i} = \lambda_{j}\}$.
Note that ${\cal C}$ is a cone in $\Rrr^{p}$.

The filter $R$ in the poset $Q_{p}$ corresponds to the subcomplex
$\Gamma$ of the complex $\Delta_{p}$ consisting
of all faces $G$ in
$\Delta_{p}$ that are contained in
the chamber ${\cal C}$.
The geometric realization of $\Gamma$ is
the intersection of the unit sphere $S^{p-1}$
and the chamber ${\cal C}$,
and hence
is homeomorphic to a $(p-1)$-dimensional ball.
Let ${\cal L}(\Gamma)$ denote the face lattice of $\Gamma$.

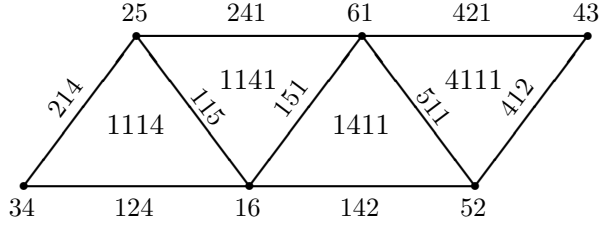
\begin{figure}
\setlength{\unitlength}{1.0mm}
\begin{center}
\begin{picture}(75,20)(0,0)

\put(0,0){\circle*{1}}
\put(15,20){\circle*{1}}
\put(30,0){\circle*{1}}
\put(45,20){\circle*{1}}
\put(60,0){\circle*{1}}
\put(75,20){\circle*{1}}

\thicklines

\put(0,0){\line(1,0){60}}
\put(15,20){\line(1,0){60}}
\put(0,0){\line(3,4){15}}
\put(30,0){\line(-3,4){15}}
\put(30,0){\line(3,4){15}}
\put(60,0){\line(-3,4){15}}
\put(60,0){\line(3,4){15}}

%% Faces:
\put(11,7){$1114$}
\put(26,13){$1141$}
\put(41,7){$1411$}
\put(56,13){$4111$}

%% Vertices:
\put(-2,-4){\small $34$}
\put(28,-4){\small $16$}
\put(58,-4){\small $52$}
\put(13,22){\small $25$}
\put(43,22){\small $61$}
\put(73,22){\small $43$}

%% Edges:
\put(12,-4){\small $124$}
\put(42,-4){\small $142$}
\put(27,22){\small $241$}
\put(57,22){\small $421$}

\put(3,9){\rotatebox{53}{\small $214$}}
\put(22,12){\rotatebox{307}{\small $115$}}
\put(33,9){\rotatebox{53}{\small $151$}}
\put(52,12){\rotatebox{307}{\small $511$}}
\put(63,9){\rotatebox{53}{\small $412$}}

\end{picture}
\end{center}
\caption{The complex $\Gamma$ corresponding to
the pointed knapsack partition $\{1,1,1,4,\underline{m}\}$.
In this figure the pointed part $\um$ has been omitted.
Observe the faces not on the boundary of $\Gamma$ correspond to
the set $V$.  Refer to Example~\ref{example_knapsack}.}
\end{figure}

A face $G$ is on the boundary of $\Gamma$
if and only if $G$ is contained in one of the hyperplanes 
in~${\cal H}$.
In other words, $G$ is on the boundary of $\Gamma$
if and only if the corresponding pointed composition
$\vec{c} = (c_{1}, \ldots, c_{k-1}, \um)$
has an entry $c_{i}$, $i<k$, such that
when $c_{i}$ is expressed uniquely as a sum of parts of $\lambda$,
two terms are equal. 
In this case
the pointed composition $\vec{c}$ does not belong to the set $V$.

For a composition $\vec{c}$ let $G$ be the associated face in
$\Gamma$. Then we have that
$$    \mu(\hz,\vec{c})
    =
      \mu_{{\cal L}(\Gamma)}(G, \ho)
    =
      \left\{ \begin{array}{c l}
             0 
                & \mbox{ if $G$ is on the boundary of $\Gamma$,} \\
             \widetilde{\chi}(\Gamma) = 0
                & \mbox{ if $G$ is the empty face,} \\
             (-1)^{\rho(G,\ho)}
                & \mbox{ otherwise,}
              \end{array} \right. $$
where the last step is Proposition~3.8.9 in~\cite{Stanley_EC_I},
proving the theorem.
\end{proof}

Combining Theorems~\ref{theorem_main}
and~\ref{theorem_M}, we have:
\begin{theorem}
Let $F$ be the filter in the
pointed integer partition poset
$I_{n}^{\bullet}$ generated by 
the pointed knapsack partition 
$\{\lambda,\um\}
   = \{\lambda_{1}, \ldots, \lambda_{p}, \um\}$
of the integer $n$.
Then the M\"obius function
$\mu(\Pi_{n}^{\bullet}(F) \cup \{\hz\})$
of the filter~$\Pi_{n}^{\bullet}(F)$
with a minimal element $\hz$ adjoined
is given by
$$
 \mu(\Pi_{n}^{\bullet}(F) \cup \{\hz\})
      =
  (-1)^{p-1}
      \cdot
  \sum_{\vec{c} \in V} \beta(\vec{c})  ,  $$
that is, $(-1)^{p-1}$ times
the number of permutations in $\Ssss_{n}$
whose descent composition belongs to the set~$V$.
\label{theorem_three}
\end{theorem}

Observe that Corollary~\ref{corollary_Dowling}
is also a consequence of Theorem~\ref{theorem_three}
using
Example~\ref{example_knapsack_r_r}.
In this case the complex $\Gamma$ is a simplex.

\addtocounter{theorem}{-4}
\begin{continuation}
Let $F$ be the filter in $I_{n}^{\bullet}$ generated
by the pointed knapsack partition $\{1,1,1,4,\um\}$.
Then 
\begin{eqnarray*}
  - \mu\left( \Pi_{n}^{\bullet}(F) \cup \{\hz\} \right)
  & = &
      \beta(1,1,1,4,\um) + \beta(1,1,5,\um) + \beta(1,1,4,1,\um) \\
  &   & 
    + \beta(1,5,1,\um) + \beta(1,4,1,1,\um) + \beta(5,1,1,\um) \\
  &   & 
    + \beta(4,1,1,1,\um) .
\end{eqnarray*}
\end{continuation}
\addtocounter{theorem}{4}

\vspace*{-6mm}

\section{Concluding remarks}
\setcounter{equation}{0}

\vspace*{-1mm}

The poset $\Pi_{n}^{\bullet}(F) \cup \{\hz\}$ raises many natural questions.
When the minimal elements of the filter $F$ have the same rank,
the poset $\Pi_{n}^{\bullet}(F) \cup \{\hz\}$ 
is graded.  One may ask if this is a shellable poset.
Similarly,
for a general filter, that is, 
when the minimal elements of the filter $F$ have different ranks,
the previous question extends to determining if
the order complex of $\Pi_{n}^{\bullet}(F) \cup \{\hz\}$
is non-pure shellable~\cite{Bjorner_Wachs}.
In the case when the poset is
not shellable, can one still determine the homology groups
of the order complex?
One goal here is to obtain a bijective proof of
Theorem~\ref{theorem_three}.

The symmetric group $\Ssss_{n}$ acts on
the order complex of $\Pi_{n}^{\bullet}(F) \cup \{\hz\}$.
This action is inherited by the homology groups.
Can this representation be determined?
When one considers a pointed knapsack partition
it is natural to conjecture that the poset is indeed
is shellable and hence the homology is concentrated
in the top homology.
Furthermore,
the dimension of the top homology is given by the
M\"obius function and
the action of the symmetric group is given by
the direct sum of Specht modules corresponding
to the skewpartitions associated with the compositions
in the set $V$.
For the pointed knapsack partition
$\{r, r, \ldots, r, \underline{r-1}\}$
corresponding to the $r$-divisible
lattice
(recall Corollary~\ref{corollary_r_divisible_lattice}),
this research program has been carried out.
See the papers~\cite{Calderbank_Hanlon_Robinson}
and~\cite{Wachs} and the references therein.

Finally, an enumerative question is to determine the number
of knapsack partitions of $n$.  This number 
seems related to the prime factor decomposition of $n$.
Techniques from analytic number theory may be required.

\vspace*{-2mm}

\section*{Acknowledgments}

We thank the referee for many helpful comments
and the MIT Mathematics Department, where this paper was completed
during the authors' sabbatical year.

\vspace*{-1mm}

\newcommand{\journal}[6]{#1, #2, #3 #4 (#5) #6.}
\newcommand{\book}[4]{#1, #2, #3, #4.}
\newcommand{\thesis}[4]{#1, #2, Doctoral dissertation, #3, #4.}
\newcommand{\preprint}[3]{#1, #2, preprint #3.}

\end{document}